\documentclass{conm-p-l}
\usepackage{amssymb}
\usepackage{graphicx}

\newtheorem{theorem}{Theorem}[section]
\newtheorem{proposition}[theorem]{Proposition}
\newtheorem{lemma}[theorem]{Lemma}
\newtheorem{corollary}[theorem]{Corollary}
\theoremstyle{definition}

\theoremstyle{remark}

\numberwithin{equation}{section}

\begin{document}

\title{A double fibration transform for complex projective space}

\author{Michael Eastwood}
\address{Mathematical Sciences Institute, Australian National University,
ACT 0200}
\curraddr{}
\email{meastwoo@member.ams.org}
\thanks{The author is supported by the Australian Research Council} 

\subjclass[2010]{Primary 32L25; Secondary 53C28}
\date{}

\dedicatory{To Sigurdur Helgason on the occasion of his eighty-fifth birthday.}

\begin{abstract}
We develop some theory of double fibration transforms where the cycle 
space is a smooth manifold and apply it to complex projective space.
\end{abstract}

\maketitle

\section{Introduction}\label{intro}
The classical {\em Penrose transform\/} is concerned with (anti)-self-dual
$4$-dimensional Riemannian manifolds. If $M$ is such a manifold then, as shown
in~\cite{AHS}, there is a canonically defined $3$-dimensional complex
manifold~$Z$, known as the {\em twistor space\/} of~$M$, that fibres over $M$
\begin{equation}\label{twistorfibration}\tau:Z\to M\end{equation}
in the sense that $\tau$ is a submersion with holomorphic fibres intrinsically
isomorphic to~${\mathbb{CP}}_1$. In fact, this construction depends only on the
conformal structure on~$M$ and the Penrose transform then identifies the
Dolbeault cohomology $H^r(Z,{\mathcal{O}}(V))$ for the various natural
holomorphic vector bundles $V$ on $Z$ with the cohomology of certain
conformally invariant elliptic complexes of linear differential operators 
on~$M$. Some typical examples are presented in~\cite{ES,H}.

The two main examples of this construction are for $M=S^4$, the flat model of
$4$-dimensional conformal geometry, and for $M={\mathbb{CP}}_2$ with its
Fubini-Study metric. In both cases, the twistor space is a well-known complex
manifold. For $S^4$ it is ${\mathbb{CP}}_3$ and for ${\mathbb{CP}}_2$ it is 
the flag manifold
$${\mathbb{F}}_{1,2}({\mathbb{C}}^3)\equiv
\{(L,P)\mid L\subset P\subset{\mathbb{C}}^3\mbox{ with }
\dim_{\mathbb{C}}L=1,\dim_{\mathbb{C}}P=2\}.$$
For ${\mathbb{CP}}_2$ the fibration is
\begin{equation}\label{CP2twistors}
{\mathbb{F}}_{1,2}({\mathbb{C}}^3)\ni(L,P)\stackrel{\tau}{\longmapsto}
L^\perp\cap P\in{\mathbb{CP}}_2,\end{equation}
where the orthogonal complement $L^\perp$ of $L$ is taken with respect to a
fixed Hermitian inner product on~${\mathbb{C}}^3$, namely the same inner
product that induces the Fubini-Study metric on ${\mathbb{CP}}_2$ as a
homogeneous space ${\mathrm{SU(3)}}/{\mathrm{S(U}}(1)\times{\mathrm{U}}(2))$. 
The Penrose transform in this setting is carried out in detail in~\cite{B,E}.

There are several options for generalising this twistor geometry of
${\mathbb{CP}}_2$ to higher dimensions. Perhaps the most obvious is to take as
twistor space the flag manifold ${\mathbb{F}}_{1,2}({\mathbb{C}}^{n+1})$ and
define $\tau:{\mathbb{F}}_{1,2}({\mathbb{C}}^{n+1})\to{\mathbb{CP}}_n$ by
$(L,P)\mapsto L^\perp\cap P$. This is the option adopted in~\cite{me}. Perhaps
a more balanced option is to take as twistor space the flag manifold
$$Z={\mathbb{F}}_{1,n}({\mathbb{C}}^{n+1})\equiv
\{(L,H)\mid L\subset H\subset{\mathbb{C}}^{n+1}\mbox{ with }
\dim_{\mathbb{C}}L=1,\dim_{\mathbb{C}}H=n\}$$
and consider the {\em double fibration}
\begin{equation}\label{balanced}\raisebox{-20pt}{\begin{picture}(60,30)
\put(0,5){\makebox(0,0){$Z$}}
\put(30,35){\makebox(0,0){$X$}}
\put(60,5){\makebox(0,0){${\mathbb{CP}}_n$}}
\put(25,30){\vector(-1,-1){18}}
\put(35,30){\vector(1,-1){18}}
\put(12,24){\makebox(0,0){$\eta$}}
\put(48,24){\makebox(0,0){$\tau$}}
\end{picture}}\end{equation}
where $X\subset{\mathbb{F}}_{1,n}({\mathbb{C}}^{n+1})\times{\mathbb{CP}}_n$ is
the {\em incidence variety\/} given by
\begin{equation}\label{incidence}
X=\{(L,H,\ell)\mid\ell\subseteq L^\perp\cap H\}\end{equation}
and the fibrations $\eta$ and $\tau$ are the {\em forgetful mappings\/},
$${\mathbb{F}}_{1,n}({\mathbb{C}}^{n+1})\ni(L,H)
\stackrel{\,\eta\,}{\longleftarrow\!\mapstochar}
(L,H,\ell)\stackrel{\,\tau\,}{\longmapsto}\ell\in{\mathbb{CP}}_n.$$
Of course, when $n=2$ the dimensions force $\eta$ to be an isomorphism and this
double fibration~(\ref{balanced}) reverts to the single
fibration~(\ref{CP2twistors}).

The aim of this article is to explain a transform on Dolbeault cohomology for
double fibrations of this type and then execute the transform in this
particular case. Then, since the Bott-Borel-Weil Theorem~\cite{bott} computes
the Dolbeault cohomology of $Z={\mathbb{F}}_{1,n}({\mathbb{C}}^{n+1})$ with
coefficients in any homogeneous vector bundle, we may draw conclusions
concerning the cohomology of various elliptic complexes on~${\mathbb{CP}}_n$.

This work was outlined at the meeting `Geometric Analysis on Euclidean and
Homogeneous Spaces' held at Tufts University in January 2012\@. The author is
grateful to the organisers, Jens Christensen, Fulton Gonzalez, and Todd Quinto,
for their invitation to speak and hospitality at that meeting and also to
Joseph~Wolf for many crucial conversations concerning this work.

\section{The general transform}\label{general}
There is a better established double fibration transform defined for 
$$\begin{picture}(60,31)(0,7)
\put(0,5){\makebox(0,0){$Z$}}
\put(30,35){\makebox(0,0){${\mathbb{X}}$}}
\put(60,5){\makebox(0,0){${\mathbb{M}}$}}
\put(25,30){\vector(-1,-1){18}}
\put(35,30){\vector(1,-1){18}}
\put(12,24){\makebox(0,0){$\mu$}}
\put(48,24){\makebox(0,0){$\nu$}}
\end{picture}$$
in which all manifolds are complex and both $\mu$ and $\nu$ are holomorphic.
Classical twistor theory, for example, is concerned with the holomorphic
correspondence
$$\begin{picture}(60,34)(0,7)
\put(0,5){\makebox(0,0){${\mathbb{CP}}_3$}}
\put(30,37){\makebox(0,0){${\mathbb{F}}_{1,2}({\mathbb{C}}^4)$}}
\put(60,5){\makebox(0,0){${\mathrm{Gr}}_2({\mathbb{C}}^4)$.}}
\put(25,30){\vector(-1,-1){18}}
\put(35,30){\vector(1,-1){18}}
\put(12,24){\makebox(0,0){$\mu$}}
\put(48,24){\makebox(0,0){$\nu$}}
\end{picture}$$
The Penrose transform in this setting is explained in~\cite{EPW} and
generalised to arbitrary holomorphic correspondences between complex flag
manifolds in~\cite{BE}. Another vast generalisation is concerned with the 
holomorphic correspondences arising from the cycle spaces of general flag 
domains as in~\cite{FHW}.

On the face of it, the double fibration (\ref{balanced}) is of a different
nature since ${\mathbb{CP}}_n$ is only to be considered as a smooth manifold.
In fact, a link will emerge with the complex correspondences and this will ease
some of the computations involved. For the moment, however, let us develop some
general machinery applicable to this {\em smooth\/} setting. This machinery is
a generalisation of the Penrose transform for a single
fibration~(\ref{twistorfibration}), which goes as follows. The only 
requirements on (\ref{twistorfibration}) are that $\tau$ should be a smooth 
submersion from a complex manifold $Z$ to a smooth manifold $M$ and that the 
fibres of $\tau$ should be compact complex submanifolds of~$Z$. 

Let us denote by $\Lambda_Z^{0,q}$ the bundle of $(0,q)$-forms on~$Z$ and by
$\bar\partial_Z:\Lambda_Z^{0,q}\to\Lambda_Z^{0,q+1}$ the
$\bar\partial$-operator on $Z$ so that
$$H^r(Z,{\mathcal{O}})\equiv 
H^r(\Gamma(Z,\Lambda_Z^{0,\bullet}),\bar\partial_Z)$$
is the Dolbeault cohomology of~$Z$. The $1$-forms along the fibres of~$\tau$, 
defined by the short exact sequence
$$0\to\tau^*\Lambda_M^1\to\Lambda_Z^1\to\Lambda_\tau^1\to 0,$$
are decomposed as $\Lambda_\tau^1=\Lambda_\tau^{0,1}\oplus\Lambda_\tau^{1,0}$
by the complex structure on these fibres and the fact that this complex
structure is acquired from that on $Z$ implies that there is a commutative
diagram 
\begin{equation}\label{acquired}\begin{array}{ccccccccc} 0&\to&\Lambda^{0,0}
&\xrightarrow{\,\bar\partial_\tau\,}&\Lambda_\tau^{0,1}
&\xrightarrow{\,\bar\partial_\tau\,}&\Lambda_\tau^{0,2}
&\xrightarrow{\,\bar\partial_\tau\,}&\cdots\\
&&\|&&\begin{array}c\uparrow\\[-10pt] \uparrow\end{array}&&
\begin{array}c\uparrow\\[-10pt] \uparrow\end{array}\\
0&\to&\Lambda^{0,0}
&\xrightarrow{\,\bar\partial_Z\,}&\Lambda_Z^{0,1}
&\xrightarrow{\,\bar\partial_Z\,}&\Lambda_Z^{0,2}
&\xrightarrow{\,\bar\partial_Z\,}&\cdots\end{array}\end{equation}
where the top row is the $\bar\partial$-complex along the fibres of~$\tau$. 
Though the notation may seem bizarre at first, let us define the bundle 
$\Lambda_\mu^{1,0}$ on $Z$ by the short exact sequence
\begin{equation}\label{bizarre}
0\to\Lambda_\mu^{1,0}\to\Lambda_Z^{0,1}\to\Lambda_\tau^{0,1}\to 0.
\end{equation}
Regarded as a filtration of~$\Lambda_Z^{0,1}$, this short exact sequence 
induces filtrations on $\Lambda_Z^{0,q}$ for all $q$ and (\ref{acquired}) 
implies that $\bar\partial_Z$ is compatible with this filtration. An immediate 
consequence is that the bundle $\Lambda_\mu^{1,0}$ acquires a holomorphic 
structure along the fibres of~$\tau$. To see this by hand, one notes that the 
composition
$$\Lambda_\mu^{1,0}\to\Lambda_Z^{0,1}\xrightarrow{\,\bar\partial_Z\,}
\Lambda_Z^{0,2}\to\Lambda_\tau^{0,2}$$
vanishes by dint of the definition (\ref{bizarre}) of $\Lambda_\mu^{1,0}$ and
the commutative diagram~(\ref{acquired}) whence the short exact sequence
$$0\to\Lambda_\mu^{2,0}\to[\ker:\Lambda_Z^{0,2}\to\Lambda_\tau^{0,2}]\to
\Lambda_\tau^{0,1}\otimes\Lambda_\mu^{1,0}\to 0$$
induced by (\ref{bizarre}) implies that $\bar\partial_Z|_{\Lambda_\mu^{1,0}}$
induces an operator 
$$\bar\partial_\tau:
\Lambda_\mu^{1,0}\to\Lambda_\tau^{0,1}\otimes\Lambda_\mu^{1,0},$$ 
as required. To see this (and much more) by machinery, one employs the spectral
sequence of the filtered complex~$\Lambda_Z^{0,\bullet}$, arriving at the
$E_0$-level 
$$\begin{picture}(260,110)
\put(2,1){\makebox(0,0){$\Lambda^{0,0}$}}
\put(2,31){\makebox(0,0){$\Lambda_\tau^{0,1}$}}
\put(2,61){\makebox(0,0){$\Lambda_\tau^{0,2}$}}
\put(2,91){\makebox(0,0){$\Lambda_\tau^{0,3}$}}
\put(72,1){\makebox(0,0){$\Lambda_\mu^{1,0}$}}
\put(72,31){\makebox(0,0){$\Lambda_\tau^{0,1}\otimes\Lambda_\mu^{1,0}$}}
\put(72,61){\makebox(0,0){$\Lambda_\tau^{0,2}\otimes\Lambda_\mu^{1,0}$}}
\put(152,1){\makebox(0,0){$\Lambda_\mu^{2,0}$}}
\put(152,31){\makebox(0,0){$\Lambda_\tau^{0,1}\otimes\Lambda_\mu^{2,0}$}}
\put(222,1){\makebox(0,0){$\Lambda_\mu^{3,0}$}} 
\put(15,0){\line(1,0){44}} 
\put(85,0){\line(1,0){54}}
\put(165,0){\line(1,0){44}} 
\put(235,0){\vector(1,0){25}}
\put(0,7){\vector(0,1){17}}
\put(0,38){\vector(0,1){16}} 
\put(0,68){\vector(0,1){16}}
\put(0,98){\vector(0,1){12}} 
\put(70,8){\vector(0,1){16}}
\put(70,38){\vector(0,1){16}} 
\put(70,68){\vector(0,1){12}}
\put(150,8){\vector(0,1){16}} 
\put(150,38){\vector(0,1){12}}
\put(220,8){\vector(0,1){12}}
\put(6,15){\makebox(0,0){\scriptsize$\bar\partial_\tau$}}
\put(6,45){\makebox(0,0){\scriptsize$\bar\partial_\tau$}}
\put(6,75){\makebox(0,0){\scriptsize$\bar\partial_\tau$}}
\put(76,15){\makebox(0,0){\scriptsize$\bar\partial_\tau$}}
\put(76,45){\makebox(0,0){\scriptsize$\bar\partial_\tau$}}
\put(156,15){\makebox(0,0){\scriptsize$\bar\partial_\tau$}}
\put(255,5){\makebox(0,0){\scriptsize$p$}}
\put(5,105){\makebox(0,0){\scriptsize$q$}}
\end{picture}$$
and, in particular, the differential
$\Lambda_\mu^{1,0}\xrightarrow{\,\bar\partial_\tau\,}
\Lambda_\tau^{0,1}\otimes\Lambda_\mu^{1,0}$. This spectral
sequence for $\Gamma(Z,\Lambda_Z^{0,\bullet})$, at the $E_1$-level, reads
$$\begin{picture}(280,110)(0,-2)
\put(2,1){\makebox(0,0){$\Gamma(M,\tau_*\Lambda_\mu^{0,0})$}}
\put(2,31){\makebox(0,0){$\Gamma(M,\tau_*^1\Lambda_\mu^{0,0})$}}
\put(2,61){\makebox(0,0){$\Gamma(M,\tau_*^2\Lambda_\mu^{0,0})$}}
\put(2,91){\makebox(0,0){$\Gamma(M,\tau_*^3\Lambda_\mu^{0,0})$}}
\put(82,1){\makebox(0,0){$\Gamma(M,\tau_*\Lambda_\mu^{1,0})$}}
\put(82,31){\makebox(0,0){$\Gamma(M,\tau_*^1\Lambda_\mu^{1,0})$}}
\put(82,61){\makebox(0,0){$\Gamma(M,\tau_*^2\Lambda_\mu^{1,0})$}}
\put(162,1){\makebox(0,0){$\Gamma(M,\tau_*\Lambda_\mu^{2,0})$}}
\put(162,31){\makebox(0,0){$\Gamma(M,\tau_*^1\Lambda_\mu^{2,0})$}}
\put(242,1){\makebox(0,0){$\Gamma(M,\tau_*\Lambda_\mu^{3,0})$}} 
\put(35,1){\vector(1,0){14}} 
\put(115,1){\vector(1,0){14}}
\put(195,1){\vector(1,0){14}} 
\put(275,1){\vector(1,0){25}}
\put(35,31){\vector(1,0){14}} 
\put(115,31){\vector(1,0){14}}
\put(35,61){\vector(1,0){14}}
\put(0,7){\line(0,1){17}}
\put(0,38){\line(0,1){16}} 
\put(0,68){\line(0,1){16}}
\put(0,98){\vector(0,1){12}} 
\put(295,6){\makebox(0,0){\scriptsize$p$}}
\put(5,105){\makebox(0,0){\scriptsize$q$}}
\end{picture}$$
where $\tau_*^q\Lambda_\mu^{p,0}$ is the $q^{\mathrm{th}}$ direct image of the
vector bundle $\Lambda_\mu^{p,0}$ with respect to its holomorphic structure in
the fibre directions. Note that, with the fibres of $\tau$ being compact, these
direct images generically define smooth vector bundles on~$M$ and certainly
this will be the case when the fibration (\ref{twistorfibration}) is
homogeneous. In any case, we have proved the following. 
\begin{theorem}\label{basic}
Suppose that $\tau:Z\to M$ is a submersion of smooth manifolds and that $Z$ has
a complex structure such that the fibres of $\tau$ are compact complex
submanifolds of~$Z$. Then the bundle
$\Lambda_\mu^{1,0}\equiv\ker:\Lambda_Z^{0,1}\to\Lambda_\tau^{0,1}$ acquires a
natural holomorphic structure along the fibres of $\tau$ and there is a
spectral sequence
$$E_1^{p,q}=\Gamma(M,\tau_*^q\Lambda_\mu^{p,0})\Longrightarrow
H^{p+q}(Z,{\mathcal{O}}).$$
\end{theorem}
This theorem only comes to life with examples in which it is possible to 
compute the direct images~$\tau_*^q\Lambda_\mu^{p,0}$. There is also a coupled 
version of the spectral sequence 
$$E_1^{p,q}=\Gamma(M,\tau_*^q\Lambda_\mu^{p,0}(V))\Longrightarrow
H^{p+q}(Z,{\mathcal{O}}(V))$$
for any holomorphic vector bundle $V$ on~$Z$. The proof is easily
modified but the added scope for interesting examples is significantly 
increased. 

For the moment, however, let us continue with generalities, firstly by
extending Theorem~\ref{basic} to cover double fibrations of the form
\begin{equation}\label{smoothdoublefibration}
\raisebox{-20pt}{\begin{picture}(60,40)
\put(0,5){\makebox(0,0){$Z$}}
\put(30,35){\makebox(0,0){$X$}}
\put(60,5){\makebox(0,0){$M$}}
\put(25,30){\vector(-1,-1){18}}
\put(35,30){\vector(1,-1){18}}
\put(12,24){\makebox(0,0){$\eta$}}
\put(48,24){\makebox(0,0){$\tau$}}
\end{picture}}\end{equation}
(of which (\ref{CP2twistors}) is typical) where $M$ is smooth and the fibres of
$\tau$ are identified by $\eta$ as compact complex submanifolds of the complex
manifold~$Z$. To do this, let us define a bundle $\Lambda_X^{0,1}$ on $X$ by
means of the short exact sequence
\begin{equation}\label{involutive}
0\to\eta^*\Lambda_Z^{1,0}\to\Lambda_X^1\to\Lambda_X^{0,1}\to 0,\end{equation}
where $\Lambda_X^1$ is the bundle of complex-valued $1$-forms on~$X$.
Geometrically, this pulls back the complex structure from $Z$ to an 
{\em involutive structure\/}~\cite{BCH} on~$X$. In particular, there is an 
induced complex of differential operators
$$0\to\Lambda_X^{0,0}\xrightarrow{\,\bar\partial_X\,}\Lambda_X^{0,1}
\xrightarrow{\,\bar\partial_X\,}\Lambda_X^{0,2}
\xrightarrow{\,\bar\partial_X\,}\cdots.$$
Comparing (\ref{involutive}) with the bundle $\Lambda_\eta^1$ of $1$-forms 
along the fibres of $\eta$
defined by the short exact sequence
$$0\to\eta^*\Lambda_Z^1\to\Lambda_X^1\to\Lambda_\eta^1\to 0$$
we see that there is a short exact sequence
$$0\to\eta^*\Lambda_Z^{0,1}\to\Lambda_X^{0,1}\to\Lambda_\eta^1\to 0.$$
The complex $\Gamma(X,\Lambda_X^{0,\bullet})$ thereby acquires a filtration,
the spectral sequence for which reads at the $E_0$-level
\begin{equation}\label{topology}
\raisebox{-17pt}{\begin{picture}(280,115)(0,-5)
\put(2,1){\makebox(0,0){$\Gamma(X,\eta^*\Lambda_Z^{0,0})$}}
\put(2,31){\makebox(0,0){$\Gamma(X,
\Lambda_\eta^1\otimes\eta^*\Lambda_Z^{0,0})$}}
\put(2,61){\makebox(0,0){$\Gamma(X,
\Lambda_\eta^2\otimes\eta^*\Lambda_Z^{0,0})$}}
\put(2,91){\makebox(0,0){$\Gamma(X,
\Lambda_\eta^3\otimes\eta^*\Lambda_Z^{0,0})$}}
\put(92,1){\makebox(0,0){$\Gamma(X,\eta^*\Lambda_Z^{0,1})$}}
\put(92,31){\makebox(0,0){$\Gamma(X,
\Lambda_\eta^1\otimes\eta^*\Lambda_Z^{0,1})$}}
\put(92,61){\makebox(0,0){$\Gamma(X,
\Lambda_\eta^2\otimes\eta^*\Lambda_Z^{0,1})$}}
\put(182,1){\makebox(0,0){$\Gamma(X,\eta^*\Lambda_Z^{0,2})$}}
\put(182,31){\makebox(0,0){$\Gamma(X,
\Lambda_\eta^1\otimes\eta^*\Lambda_Z^{0,2})$}}
\put(262,1){\makebox(0,0){$\Gamma(X,\eta^*\Lambda_Z^{0,3})$}}
\put(35,1){\line(1,0){24}} \put(125,1){\line(1,0){24}}
\put(215,1){\line(1,0){14}} \put(295,1){\vector(1,0){13}}
\put(0,7){\vector(0,1){17}} \put(0,38){\vector(0,1){16}}
\put(0,68){\vector(0,1){16}} \put(0,98){\vector(0,1){12}}
\put(90,7){\vector(0,1){17}} \put(90,38){\vector(0,1){16}}
\put(90,68){\vector(0,1){12}} \put(180,7){\vector(0,1){17}}
\put(180,38){\vector(0,1){12}} \put(6,15){\makebox(0,0){\scriptsize$d_\eta$}}
\put(6,45){\makebox(0,0){\scriptsize$d_\eta$}}
\put(6,75){\makebox(0,0){\scriptsize$d_\eta$}}
\put(96,15){\makebox(0,0){\scriptsize$d_\eta$}}
\put(96,45){\makebox(0,0){\scriptsize$d_\eta$}}
\put(186,15){\makebox(0,0){\scriptsize$d_\eta$}}
\put(303,6){\makebox(0,0){\scriptsize$p$}}
\put(5,105){\makebox(0,0){\scriptsize$q$}}
\end{picture}}\end{equation}
where $d_\eta:\Lambda_\eta^q\otimes\eta^*\Lambda^{0,p}\to
\Lambda_\eta^{q+1}\otimes\eta^*\Lambda^{0,p}$ is the exterior derivative along 
the fibres of $\eta$ coupled with the pullback bundle $\eta^*\Lambda^{0,p}$. 
Notice that such a coupling 
$$d_\eta:\eta^*V\to\Lambda_\eta^1\otimes\eta^*V\quad\mbox{and hence}\quad
d_\eta:\Lambda_\eta^q\otimes\eta^*V\to\Lambda_\eta^{q+1}\otimes\eta^*V$$
is valid for any smooth vector bundle $V$ on $Z$ because the pullback
$\eta^*V$ may be defined by transition functions that are constant along the
fibres, hence annihilated by~$d_\eta$. When the fibres of $\eta$ are
contractible, this is exactly the setting in which Buchdahl's
theorem~\cite{npb} applies and we deduce the following.
\begin{proposition}
Suppose that the fibres of $\eta:X\to Z$ are contractible. Then 
$$0\to\Gamma(Z,V)\to\Gamma(X,\eta^*V)
\xrightarrow{\,d_\eta\,}\Gamma(X,\Lambda_\eta^1\otimes\eta^*V)
\xrightarrow{\,d_\eta\,}\Gamma(X,\Lambda_\eta^2\otimes\eta^*V)
\xrightarrow{\,d_\eta\,}\cdots$$
is exact for any smooth vector bundle $V$ on~$Z$.
\end{proposition}
In this case, our spectral sequence (\ref{topology}) collapses at the
$E_1$-level and we have proved the following.
\begin{proposition}\label{pullback}
Suppose that the fibres of $\eta:X\to Z$ are contractible. Then 
$$H^r(Z,{\mathcal{O}})\cong 
H^r(\Gamma(X,\Lambda_X^{0,\bullet}),\bar\partial_X)\enskip\mbox{for all}
\enskip r=0,1,2,\ldots.$$
\end{proposition}
In fact, for the double fibration~(\ref{balanced}), the fibres of $\eta$ are
not contractible and in \S\ref{particular} we shall have to revisit the
spectral sequence (\ref{topology}) to relate the Dolbeault cohomology
$H^r(Z,{\mathcal{O}})$ with the {\em involutive cohomology\/}
$H^r(\Gamma(X,\Lambda_X^{0,\bullet}),\bar\partial_X)$.

Nevertheless, we may deal with the fibration $\tau:X\to M$ exactly as in our 
proof of Theorem~\ref{basic}. Specifically, we define a bundle 
$\Lambda_\mu^{1,0}$ on $X$ by the exact sequence
\begin{equation}\label{bizarro}
0\to\Lambda_\mu^{1,0}\to\Lambda_X^{0,1}\to\Lambda_\tau^{0,1}\to 0\end{equation}
and employ the spectral sequence of the corresponding filtered complex 
$\Lambda_X^{0,\bullet}$ to conclude that the following theorem holds.
\begin{theorem}\label{advanced}
Suppose that
\begin{equation}\tag{\ref{smoothdoublefibration}}
\raisebox{-20pt}{\begin{picture}(60,40)
\put(0,5){\makebox(0,0){$Z$}}
\put(30,35){\makebox(0,0){$X$}}
\put(60,5){\makebox(0,0){$M$}}
\put(25,30){\vector(-1,-1){18}}
\put(35,30){\vector(1,-1){18}}
\put(12,24){\makebox(0,0){$\eta$}}
\put(48,24){\makebox(0,0){$\tau$}}
\end{picture}}\end{equation}
is a double fibration of smooth manifolds such that 
\begin{itemize}
\item $Z$ is a complex manifold,
\item the fibres of $\tau$ are embedded by $\eta$ as compact complex 
submanifolds of~$Z$.
\end{itemize}
Then the bundle~$\Lambda_\mu^{1,0}$, defined as the middle cohomology of the
complex 
\begin{equation}\label{middlecohomology}
0\to\eta^*\Lambda_Z^{1,0}\to\Lambda_X^1\to\Lambda_\tau^{0,1}\to 0,
\end{equation}
acquires a natural holomorphic structure along the fibres of\/ $\tau$ 
and there is a spectral sequence
\begin{equation}\label{SS}
E_1^{p,q}=\Gamma(M,\tau_*^q\Lambda_\mu^{p,0})\Longrightarrow
H^{p+q}(\Gamma(X,\Lambda_X^{0,\bullet}),\bar\partial_X).\end{equation}
\end{theorem}
\begin{corollary}\label{contractiblecorollary}
If, in addition, the fibres of $\eta$ are contractible, then 
$$E_1^{p,q}=\Gamma(M,\tau_*^q\Lambda_\mu^{p,0})\Longrightarrow
H^{p+q}(Z,{\mathcal{O}}).$$ 
\end{corollary}
\begin{proof}Immediate from Proposition~\ref{pullback}.
\end{proof}
\noindent In \S\ref{flag} we shall present an example for which the fibres of 
$\eta$ are, indeed, contractible and to which 
Corollary~\ref{contractiblecorollary} applies.

Before we continue, let us glance ahead to \S\ref{particular} in which the
first thing we do is use (\ref{topology}) to deal with the topology along the
fibres of $\eta$ for the double fibration~(\ref{balanced}). Another thing we
need in order to apply Theorem~\ref{advanced} is a computation of the direct
images $\tau_*^q\Lambda_\mu^{p,0}$ as homogeneous bundles on~${\mathbb{CP}}_n$.
This computation is best viewed in the light of a geometric interpretation of
$\Lambda_\mu^{1,0}$ as follows.
 
Suppose that $M$ is a totally real submanifold of a complex 
manifold~${\mathbb{M}}$ such that the double fibration 
(\ref{smoothdoublefibration}) embeds as
\begin{equation}\label{complexify}\raisebox{-20pt}{\begin{picture}(60,40)
\put(0,5){\makebox(0,0){$Z$}}
\put(30,35){\makebox(0,0){$X$}}
\put(60,5){\makebox(0,0){$M$}}
\put(25,30){\vector(-1,-1){18}}
\put(35,30){\vector(1,-1){18}}
\put(12,24){\makebox(0,0){$\eta$}}
\put(48,24){\makebox(0,0){$\tau$}}
\end{picture}}
\qquad\mbox{\Large$\hookrightarrow$}\qquad
\raisebox{-20pt}{\begin{picture}(60,40)
\put(0,5){\makebox(0,0){$Z$}}
\put(30,35){\makebox(0,0){${\mathbb{X}}$}}
\put(60,5){\makebox(0,0){${\mathbb{M}}\,,$}}
\put(25,30){\vector(-1,-1){18}}
\put(35,30){\vector(1,-1){18}}
\put(12,24){\makebox(0,0){$\mu$}}
\put(48,24){\makebox(0,0){$\nu$}}
\end{picture}}\end{equation}
where the ambient double fibration is in the holomorphic category and the 
fibres of $\nu$ coincide with the fibres of $\tau$ over~$M$.
\begin{proposition}\label{bizarre_notation_explained}
Under these circumstances the bundle $\Lambda_\mu^{1,0}$ of $(1,0)$-forms along
the fibres of $\mu$ coincides, when restricted to $X\subset{\mathbb{X}}$, with
the bundle already denoted in the same way and defined as the middle
cohomology of~\eqref{middlecohomology}.
\end{proposition}
\begin{proof} If we write 
$$n=\dim_{\mathbb{C}}Z\qquad m=\dim_{\mathbb{R}}M\qquad
s=\dim_{\mathbb{C}}(\mbox{fibres of }\tau),$$
then $\dim_{\mathbb{R}}X=m+2s$ and $X$ has real codimension $2(n-s)$
in~$Z\times M$. This is the same as the real codimension of ${\mathbb{X}}$ in
$Z\times{\mathbb{M}}$ and it follows that the complexified conormal bundle
${\mathcal{C}}$ of $X$ in $Z\times M$ coincides with the restriction to $X$ of
the complexified conormal bundle of ${\mathbb{X}}$ in~$Z\times{\mathbb{M}}$.
Hence it splits as ${\mathcal{C}}={\mathcal{C}}^{0,1}\oplus{\mathcal{C}}^{1,0}$
in line with the complex structure on the ambient double fibration. For any 
double fibration there is a basic commutative diagram with exact rows and
columns, which in the case of (\ref{smoothdoublefibration}) looks as follows.
$$\begin{array}{ccccccccc}
&&0&&0\\
&&\uparrow&&\uparrow\\
&&\Lambda_\tau^1&=&\Lambda_\tau^1\\
&&\uparrow&&\uparrow\\
0&\to&\eta^*\Lambda_Z^1&\to&\Lambda_X^1&\to&\Lambda_\eta^1&\to&0\\
&&\uparrow&&\uparrow&&\|\\
0&\to&{\mathcal{C}}&\to&\tau^*\Lambda_M^1&\to&\Lambda_\eta^1&\to&0\\
&&\uparrow&&\uparrow\\
&&0&&0
\end{array}$$
But we have just observed that the left hand column has the additional feature 
that it splits 
$$\begin{array}{ccc}
\Lambda_\tau^1&=&\Lambda_\tau^{0,1}\oplus\Lambda_\tau^{1,0}\\
\uparrow&&\uparrow\hspace{30pt}\uparrow\\
\eta^*\Lambda_Z^1&&\eta^*\Lambda_Z^{0,1}\oplus\eta^*\Lambda_Z^{1,0}\\
\uparrow&&\uparrow\hspace{30pt}\uparrow\\
{\mathcal{C}}&=&\,{\mathcal{C}}^{0,1}\oplus\,{\mathcal{C}}^{1,0}
\end{array}$$
in line with the ambient complex structure. Hence we obtain the diagram
$$\begin{array}{ccccccccc}
&&0&&0\\
&&\uparrow&&\uparrow\\
&&\Lambda_\tau^{0,1}&=&\Lambda_\tau^{0,1}\\
&&\uparrow&&\uparrow\\
0&\to&\eta^*\Lambda_Z^{0,1}&\to&
\Lambda_X^1/\eta^*\Lambda_Z^{1,0}&\to&\Lambda_\eta^1&\to&0\\
&&\uparrow&&\uparrow&&\|\\
0&\to&{\mathcal{C}}^{0,1}&\to&
\tau^*\Lambda_M^1/{\mathcal{C}}^{1,0}&\to&\Lambda_\eta^1&\to&0\\
&&\uparrow&&\uparrow\\
&&0&&0
\end{array}$$
and it follows from (\ref{involutive}) and (\ref{bizarro}) that
$\Lambda_\mu^{1,0}=\tau^*\Lambda_M^1/{\mathcal{C}}^{1,0}$. On the other hand,
since $M\hookrightarrow{\mathbb{M}}$ is totally real, we may identify
$\Lambda_M^1$ with $\Lambda_{\mathbb{M}}^{1,0}$ along $M$ and therefore
$\tau^*\Lambda_M^1$ with $\nu^*\Lambda_{\mathbb{M}}^{1,0}$ along
$X\subset{\mathbb{X}}$, at which point the basic diagram on~${\mathbb{X}}$
\begin{equation}\label{final_basic_diagram}\begin{array}{ccccccccc} &&0&&0\\
&&\uparrow&&\uparrow\\
&&\Lambda_\nu^{1,0}&=&\Lambda_\nu^{1,0}\\
&&\uparrow&&\uparrow\\
0&\to&\mu^*\Lambda_Z^{1,0}&\to&\Lambda_{\mathbb{X}}^{1,0}&\to&
\Lambda_\mu^{1,0}&\to&0\\
&&\uparrow&&\uparrow&&\|\\
0&\to&{\mathcal{C}}^{1,0}&\to&\nu^*\Lambda_{\mathbb{M}}^{1,0}&\to&
\Lambda_\mu^{1,0}&\to&0\\
&&\uparrow&&\uparrow\\
&&0&&0
\end{array}\end{equation}
for the ambient double fibration in the holomorphic setting finishes the proof.
\end{proof}
Finally, it is left to the reader also to check that the holomorphic structure
for the bundle $\Lambda_\mu^{1,0}$ on $X$ along the fibres of~$\tau$ coincides
with the standard holomorphic structure along the fibres of $\mu$ for the
bundle $\Lambda_\mu^{1,0}$ on ${\mathbb{X}}$ when restricted to
$X\hookrightarrow{\mathbb{X}}$.

In summary, for a double fibration of the form~(\ref{smoothdoublefibration}),
firstly we have a spectral sequence (\ref{topology}) that can be used to
interpret Dolbeault cohomology on $Z$ in terms of involutive cohomology on $X$,
secondly another spectral sequence (\ref{SS}) that can be used to interpret the
involutive cohomology on $X$ in terms of smooth data on $M$ and, thirdly, in
case that (\ref{smoothdoublefibration}) complexifies as~(\ref{complexify}), a
geometric interpretation of the bundles $\Lambda_\mu^{p,0}$ occurring in this
spectral sequence. In the following section, we shall see that this is just
what we need to operate the transform onto complex projective space starting
with the double fibration~(\ref{balanced}).

\section{A particular transform}\label{particular}
This section is entirely concerned with the double fibration~(\ref{balanced}),
which will be dealt with mainly by means of Theorem~\ref{advanced}. But, as
foretold in~\S\ref{general}, the first thing we should do is deal with the
topology of the fibres of~$\eta$. 
\begin{proposition}\label{thefibres}
For the double fibration \eqref{balanced}
\begin{itemize}
\item the fibres of\/ $\eta$ are isomorphic to\/~${\mathbb{CP}}_{n-2}$ as 
smooth manifolds,
\item the fibres of\/ $\tau$ are isomorphic to 
${\mathbb{F}}_{1,n-1}({\mathbb{C}}^n)$ as complex manifolds.
\end{itemize}
\end{proposition}
\begin{proof} It is useful to draw a picture in ${\mathbb{CP}}_n$ of the
incidence variety~(\ref{incidence}) (although, of course, this is a picture over
the reals in case $n=3$). 
$$\begin{picture}(200,212)
\put(100,105){\makebox(0,0){\rotatebox{-90}
{\includegraphics[width=220pt]{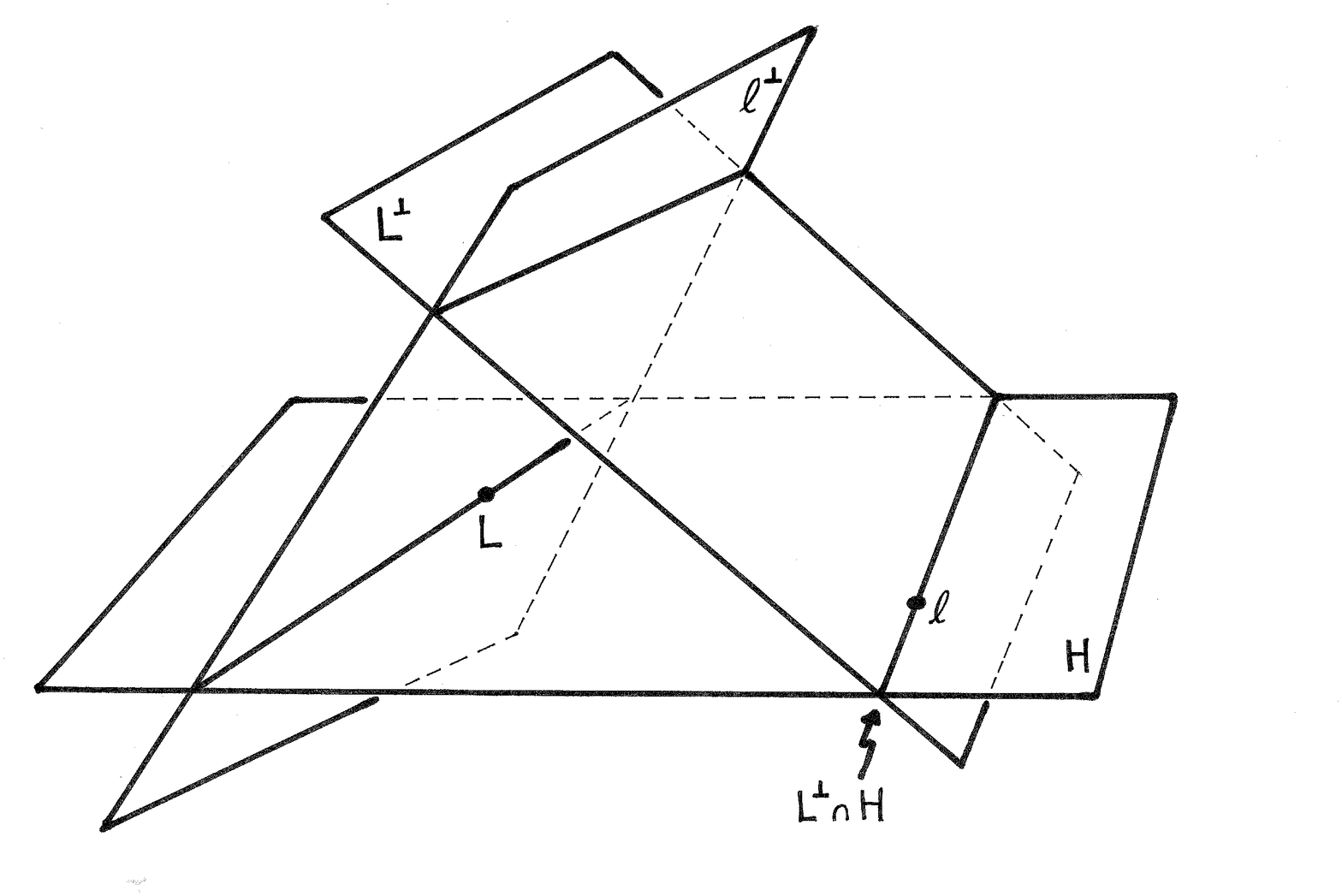}}}}
\end{picture}$$
There are two points $L$ and $\ell$ and three hyperplanes $L^\perp$,
$\ell^\perp$, and~$H$. Since $L^\perp\cap H$ is the intersection of two
hyperplanes in ${\mathbb{CP}}_n$ it is intrinsically~${\mathbb{CP}}_{n-2}$ and,
since the mapping $\eta$ from this configuration to
${\mathbb{F}}_{1,n}({\mathbb{C}}^{n+1})$ forgets everything but~$L\in H$, 
we have identified its fibres with~${\mathbb{CP}}_{n-2}$. On the other hand, 
if $\ell$ is fixed, then the rest of the configuration may be constructed by 
choosing an arbitrary point $L\in\ell^\perp$ and an arbitrary hyperplane in 
$\ell^\perp$ passing through $L$, defining H as the join of this hyperplane 
with~$\ell$.
\end{proof}
Examining this configuration also shows how the double fibration
(\ref{balanced}) may be naturally complexified to obtain~(\ref{complexify}).
One simply allows the point $\ell\in{\mathbb{CP}}_n$ and the hyperplane
$\ell^\perp\in{\mathbb{CP}}_n^*$ to become unrelated save for retaining that
$\ell\not\in\ell^\perp$. More precisely, let
$${\mathbb{M}}\equiv
\{(\ell,h)\in{\mathbb{CP}}_n\times{\mathbb{CP}}_n^*\mid\ell\not\in h\}=
{\mathbb{F}}_{1}({\mathbb{C}}^{n+1})\!\times\!{\mathbb{F}}_{n}({\mathbb{C}}^{n+1})
\setminus{\mathbb{F}}_{1,n}({\mathbb{C}}^{n+1})$$
with ${\mathbb{CP}}_n\equiv M\hookrightarrow{\mathbb{M}}$ given by 
$\ell\mapsto(\ell,\ell^\perp)$, where the orthogonal complement is taken with 
respect to a fixed Hermitian inner product on~${\mathbb{C}}^{n+1}$. If we set
$${\mathbb{X}}\equiv\{(L,H,\ell,h)\mid L\subset h\mbox{ and }\ell\subset H\}$$
then clearly this extends $X$ in~(\ref{incidence}): the geometry is exactly the
same except that $\ell^\perp$ is replaced by the less constrained
hyperplane~$h$. An advantage of the complexified double fibration is 
that it is homogeneous under the action of ${\mathrm{GL}}(n+1,{\mathbb{C}})$:
\small
\begin{equation}\label{homogeneous}
\raisebox{-100pt}{\begin{picture}(300,200)(10,0)
\put(60,40){\makebox(0,0){$\setlength{\arraycolsep}{3pt}
\renewcommand{\arraystretch}{.8}
{\mathrm{GL}}(n+1,{\mathbb{C}})
\Bigg/
\left\{\left[\begin{array}{c|c|ccc|c}
\rule{0pt}{8pt}*&0&*&\cdots&*&*\\ \hline
\rule{0pt}{8pt}*&*&*&\cdots&*&*\\ \hline
\rule{0pt}{8pt}*&0&*&\cdots&*&*\\[-2pt]
\vdots&\vdots&\vdots &&\vdots&\vdots\\
\rule{0pt}{8pt}*&0&*&\cdots&*&*\\ \hline
\rule{0pt}{8pt}0&0&0&\cdots&0&*
\end{array}\right]\right\}$}}
\put(150,160){\makebox(0,0){$\setlength{\arraycolsep}{3pt}
\renewcommand{\arraystretch}{.8}
{\mathrm{GL}}(n+1,{\mathbb{C}})
\Bigg/
\left\{\left[\begin{array}{c|c|ccc|c}
\rule{0pt}{8pt}*&0&0&\cdots&0&0\\ \hline
\rule{0pt}{8pt}0&*&*&\cdots&*&*\\ \hline
\rule{0pt}{8pt}0&0&*&\cdots&*&*\\[-2pt]
\vdots&\vdots&\vdots &&\vdots&\vdots\\
\rule{0pt}{8pt}0&0&*&\cdots&*&*\\ \hline
\rule{0pt}{8pt}0&0&0&\cdots&0&*
\end{array}\right]\right\}$}}
\put(240,40){\makebox(0,0){$\setlength{\arraycolsep}{3pt}
\renewcommand{\arraystretch}{.8}
{\mathrm{GL}}(n+1,{\mathbb{C}})
\Bigg/
\left\{\left[\begin{array}{c|c|ccc|c}
\rule{0pt}{8pt}*&0&0&\cdots&0&0\\ \hline
\rule{0pt}{8pt}0&*&*&\cdots&*&*\\ \hline
\rule{0pt}{8pt}0&*&*&\cdots&*&*\\[-2pt]
\vdots&\vdots&\vdots &&\vdots&\vdots\\
\rule{0pt}{8pt}0&*&*&\cdots&*&*\\ \hline
\rule{0pt}{8pt}0&*&*&\cdots&*&*
\end{array}\right]\right\}$}}
\put(130,120){\vector(-1,-1){40}}
\put(170,120){\vector(1,-1){40}}
\put(108,105){\makebox(0,0){\large$\mu$}}
\put(192,105){\makebox(0,0){\large$\nu$}}
\end{picture}}\end{equation}
\normalsize
Before exploiting this homogeneity, however, there is an immediate
consequence of Proposition~\ref{thefibres}, as follows.
\begin{proposition}\label{computation_involutive_cohomology}
Concerning the double fibration~\eqref{balanced}, there are canonical
isomorphisms
$$H^r(\Gamma(X,\Lambda_X^{0,\bullet}),\bar\partial_X)={\mathbb{C}}\quad
\mbox{for}\enskip r=0,2,4,6,\cdots, 2n-4,$$
and the cohomology in other degrees vanishes.
\end{proposition}
\begin{proof} {From} Proposition~\ref{thefibres} and the well-known de~Rham
cohomology of complex projective space~\cite{BT}, it follows from 
(\ref{topology}) that the $E_1$-level of this spectral sequence is isomorphic 
to
\begin{equation}\label{theE1level}
\raisebox{-19pt}{\begin{picture}(280,95)(0,0)
\put(0,1){\makebox(0,0){$\Gamma(Z,\Lambda^{0,0})$}}
\put(0,21){\makebox(0,0){$0$}}
\put(0,41){\makebox(0,0){$\Gamma(Z,\Lambda^{0,0})$}}
\put(0,61){\makebox(0,0){$0$}}
\put(0,81){\makebox(0,0){$\Gamma(Z,\Lambda^{0,0})$}}
\put(80,1){\makebox(0,0){$\Gamma(Z,\Lambda^{1,0})$}}
\put(80,21){\makebox(0,0){$0$}}
\put(80,41){\makebox(0,0){$\Gamma(Z,\Lambda^{1,0})$}}
\put(80,61){\makebox(0,0){$0$}}
\put(80,81){\makebox(0,0){$\Gamma(Z,\Lambda^{1,0})$}}
\put(160,1){\makebox(0,0){$\Gamma(Z,\Lambda^{2,0})$}}
\put(160,21){\makebox(0,0){$0$}}
\put(160,41){\makebox(0,0){$\Gamma(Z,\Lambda^{2,0})$}}
\put(160,61){\makebox(0,0){$0$}}
\put(160,81){\makebox(0,0){$\cdots$}}
\put(240,1){\makebox(0,0){$\Gamma(Z,\Lambda^{3,0})$}} 
\put(240,21){\makebox(0,0){$0$}} 
\put(240,41){\makebox(0,0){$\cdots$}} 
\put(35,1){\vector(1,0){14}} 
\put(115,1){\vector(1,0){14}}
\put(195,1){\vector(1,0){14}} 
\put(35,21){\vector(1,0){14}} 
\put(115,21){\vector(1,0){14}}
\put(195,21){\vector(1,0){14}} 
\put(35,41){\vector(1,0){14}} 
\put(115,41){\vector(1,0){14}}
\put(195,41){\vector(1,0){14}} 
\put(35,61){\vector(1,0){14}} 
\put(115,61){\vector(1,0){14}}
\put(35,81){\vector(1,0){14}} 
\put(115,81){\vector(1,0){14}}
\put(275,1){\vector(1,0){25}}
\put(0,7){\line(0,1){7}}
\put(0,27){\line(0,1){7}} 
\put(0,47){\line(0,1){7}}
\put(0,67){\line(0,1){7}}
\put(0,87){\vector(0,1){12}} 
\put(295,6){\makebox(0,0){\scriptsize$p$}}
\put(5,95){\makebox(0,0){\scriptsize$q$}}
\end{picture}}\end{equation}
But, the fibres of $\eta$ are not only isomorphic to ${\mathbb{CP}}_{n-2}$ as
smooth manifolds but as K\"ahler manifolds---the fixed Hermitian inner product
on ${\mathbb{C}}^{n+1}$ endows each fibre with a canonical K\"ahler metric. In
particular, the K\"ahler form and its exterior powers provide an explicit
basis for the de~Rham cohomology and therefore this identification of the
$E_1$-level becomes canonical. Now, as a very special case of the
Bott-Borel-Weil Theorem~\cite{bott}, the cohomology of each row of 
(\ref{theE1level}) is concentrated in zeroth position where it is canonically 
identified with~${\mathbb{C}}$. As this spectral sequence converges to 
$H^{p+q}(\Gamma(X,\Lambda_X^{0,\bullet}),\bar\partial_X)$, the proof is 
complete.\end{proof}

Now we come to the task of interpreting the spectral sequence~(\ref{SS}). As
already mentioned in \S\ref{general} we shall use
Proposition~\ref{bizarre_notation_explained} and the complexified double 
fibration
$$\raisebox{-20pt}{\begin{picture}(55,40)
\put(0,5){\makebox(0,0){${\mathbb{F}}_{1,n}({\mathbb{C}}^{n+1})$}}
\put(30,35){\makebox(0,0){$X$}}
\put(60,5){\makebox(0,0){${\mathbb{CP}}_n$}}
\put(25,30){\vector(-1,-1){18}}
\put(35,30){\vector(1,-1){18}}
\put(12,24){\makebox(0,0){$\eta$}}
\put(48,24){\makebox(0,0){$\tau$}}
\end{picture}}
\qquad\mbox{\Large$\hookrightarrow$}\hspace{30pt}
\raisebox{-20pt}{\begin{picture}(180,40)
\put(0,5){\makebox(0,0){${\mathbb{F}}_{1,n}({\mathbb{C}}^{n+1})$}}
\put(30,35){\makebox(0,0){${\mathbb{X}}$}}
\put(55,5){\makebox(0,0)[l]{${\mathbb{M}}=
\{(\ell,h)\in{\mathbb{CP}}_n\times{\mathbb{CP}}_n^*\mid\ell\not\in h\}$}}
\put(25,30){\vector(-1,-1){18}}
\put(35,30){\vector(1,-1){18}}
\put(12,24){\makebox(0,0){$\mu$}}
\put(48,24){\makebox(0,0){$\nu$}}
\end{picture}}$$
to identify the direct images $\tau_*^q\Lambda_\mu^{0,p}$. This, in turn, will
be facilitated by the fact that the complexification is 
${\mathrm{GL}}(n+1,{\mathbb{C}})$-homogeneous as in~(\ref{homogeneous}).

For simplicity, we shall now restrict to the case $n=3$, the general case 
being only notationally more awkward. Adapting the notation of~\cite{mge}, 
the irreducible homogeneous vector bundles on ${\mathbb{M}}$ may be denoted
$$(a\,\|\,b,c,d)\quad\mbox{for integers}\enskip a,b,c,d\enskip\mbox{with}
\enskip b\leq c\leq d.$$
For example, the holomorphic cotangent bundle is
\begin{equation}\label{Lambda_M^1}
(-1\,\|\,0,0,1)\oplus(1\,\|\,{-1},0,0),\end{equation}
being the analytic continuation of the bundle 
$\Lambda_M^1=\Lambda_M^{0,1}\oplus\Lambda_M^{1,0}$ on~$M$. 
Similarly, the irreducible homogeneous vector bundles on ${\mathbb{X}}$ are
necessarily line bundles and may be denoted
$$(a\,\|\,b\,|\,c\,|\,d)\quad\mbox{for arbitrary integers}\enskip a,b,c,d.$$
By carefully unravelling the meaning of these symbols in terms of weights, one 
can check that the bundle $\Lambda_\mu^{1,0}$ is reducible and
\begin{equation}\label{key_bundle}\Lambda_\mu^{1,0}=\begin{array}{c}
(-1\,\|\,0\,|\,0\,|\,1)+(-1\,\|\,0\,|\,1\,|\,0)\\ \oplus\\
(1\,\|\,{-1}\,|\,0\,|\,0)+(1\,\|\,0\,|\,{-1}\,|\,0)
\end{array}\end{equation}
where $(-1\,\|\,0\,|\,0\,|\,1)+(-1\,\|\,0\,|\,1\,|\,0)$, for example, means
that this is a rank $2$ bundle with composition factors as indicated,
equivalently that there is an exact sequence
$$0\to(-1\,\|\,0\,|\,1\,|\,0)\to(-1\,\|\,0\,|\,0\,|\,1)+(-1\,\|\,0\,|\,1\,|\,0)
\to(-1\,\|\,0\,|\,0\,|\,1)\to 0.$$
The procedure for computing direct images is explained in~\cite{mge} and here 
we find
$$\nu_*(-1\,\|\,0\,|\,0\,|\,1)=(-1\,\|\,0,0,1)\qquad
\nu_*(1\,\|\,{-1}\,|\,0\,|\,0)=(1\,\|\,{-1},0,0)$$
with all other direct images vanishing (e.g.~$(-1\,\|\,0\,|\,1\,|\,0)$ is
singular along the fibres of~$\nu$). Bearing in mind that the fibres of $\nu$ 
coincide with those of $\tau$ over~$M$, we have proved the following.
\begin{lemma}
For the double fibration \eqref{balanced} and $\Lambda_\mu^{1,0}$ defined on 
$X$ by the exact sequence~\eqref{bizarro}, we have
$$\tau_*\Lambda_\mu^{1,0}=\Lambda_M^1\quad\mbox{and all higher direct images 
vanish.}$$
\end{lemma}
{From} (\ref{key_bundle}) and the algorithms in \cite{mge} the higher forms are
$$\Lambda_\mu^{2,0}=(-2\,\|\,0\,|\,1\,|\,1)\oplus
\bigg[(0\,\|\,{-1}\,|\,0\,|\,1)+
\!\!\!\begin{array}{c}
(0\,\|\,0\,|\,{-1}\,|\,1)\\ \oplus\\
(0\,\|\,{-1}\,|\,1\,|\,0)
\end{array}\!\!\!
+(0\,\|\,0\,|\,0\,|\,0)\bigg]
\oplus (2\,\|\,{-1}\,|\,{-1}\,|\,0)$$
and therefore
$$\nu_*\Lambda_\mu^{2,0}=(-2\,\|\,0,1,1)\oplus(0\,\|\,{-1},0,1)\oplus
(0\,\|\,0,0,0)\oplus(2\,\|\,{-1},{-1},0)=\Lambda_M^2$$
with all higher direct images vanishing. Next,
$$\Lambda_\mu^{3,0}=\begin{array}{c}
(-1\,\|\,{-1}\,|\,1\,|\,1)+(-1\,\|\,0\,|\,0\,|\,1)\\ \oplus\\
(1\,\|\,{-1}\,|\,{-1}\,|\,1)+(1\,\|\,{-1}\,|\,0\,|\,0)
\end{array}$$
whence
$$\nu_*\Lambda_\mu^{3,0}=\begin{array}{c}
(-1\,\|\,{-1},1,1)\oplus(-1\,\|\,0,0,1)\\ \oplus\\
(1\,\|\,{-1},{-1},1)\oplus(1\,\|\,{-1},0,0)
\end{array}=\begin{array}{c}\Lambda_M^{1,2}\\ \oplus\\
\Lambda_M^{2,1}\end{array}$$
with all higher direct images vanishing. Finally,
$$\Lambda_\mu^{4,0}=(0\,\|\,{-1}\,|\,0\,|\,1)\enskip\implies\enskip
\nu_*\Lambda_\mu^{4,0}=(0\,\|\,{-1},0,1)=\Lambda_{M,\perp}^{2,2},$$
where $\Lambda_{M,\perp}^{2,2}$ denotes the $(2,2)$-forms orthogonal to 
$\kappa\wedge\kappa$ where $\kappa$ is the K\"ahler form on 
$M={\mathbb{CP}}_3$. Again, the higher direct images vanish.

Feeding all this information into the spectral sequence of
Theorem~\ref{advanced} causes it to collapse to an identification of the
involutive cohomology $H^r(\Gamma(X,\Lambda_X^{0,\bullet}),\bar\partial_X)$ as 
the global cohomology of the elliptic complex 
\begin{equation}\label{elliptic_complex}0\to\Lambda^0\xrightarrow{d}
\Lambda^1\xrightarrow{d}
\Lambda^2\to
\begin{array}{c}\Lambda^{1,2}\\[-2pt] \oplus\\ 
\Lambda^{2,1}\end{array}
\to\Lambda_\perp^{2,2}
\to 0\end{equation}
on ${\mathbb{CP}}_3$ and from
Proposition~\ref{computation_involutive_cohomology} we deduce the following.
\begin{theorem}
The complex \eqref{elliptic_complex} is exact on ${\mathbb{CP}}_3$ except at 
$\Lambda^0$ and $\Lambda^2$, where its cohomology is canonically identified 
with~${\mathbb{C}}$.  
\end{theorem}
\noindent In fact, the K\"ahler form on ${\mathbb{CP}}_3$ generates the
cohomology at~$\Lambda^2$. It is interesting to compare
(\ref{elliptic_complex}) with the complex that emerges from the Penrose
transform of $H^r({\mathbb{F}}_{1,2}({\mathbb{C}}^4),{\mathcal{O}})$ under the
submersion ${\mathbb{F}}_{1,2}({\mathbb{C}}^4)\to{\mathbb{CP}}_3$ as computed
in~\cite{me}, namely
$$0\to\Lambda^0\xrightarrow{d}
\Lambda^1\to
\begin{array}{c}\Lambda^{0,2}\\[-2pt] \oplus\\ 
\Lambda_\perp^{1,1}\end{array}
\to\Lambda_\perp^{1,2}
\to 0,$$
which is exact except for the constants at~$\Lambda^0$. The complex 
(\ref{elliptic_complex}) is better balanced with respect to type, as one would 
expect.

As a simple variation on this theme, one can consider a similar transform for
the Dolbeault cohomology of $Z={\mathbb{F}}_{1,3}({\mathbb{C}}^4)$ but having
coefficients in any complex homogeneous line bundle or, indeed, vector bundle
on~$Z$. Following the notation of~\cite{mge}, let us next consider the
homogeneous line bundle $(1\,|\,0,0\,|\,0)$ on~$Z$. The only additional
difficulty that must be addressed is that ${\mathbb{F}}_{1,3}({\mathbb{C}}^4)$,
as it appears in~(\ref{homogeneous}), is not written in standard form. 
Specifically, we have 
$$\setlength{\arraycolsep}{3pt}
\renewcommand{\arraystretch}{.8}
{\mathrm{GL}}(4,{\mathbb{C}})\Big/
\left\{\left[\begin{array}{cccc}
*&0&*&*\\
*&*&*&*\\
*&0&*&*\\
0&0&0&*\end{array}\right]\right\}
\quad\mbox{rather than}\quad
{\mathrm{GL}}(4,{\mathbb{C}})\Big/
\left\{\left[\begin{array}{cccc}
*&*&*&*\\
0&*&*&*\\
0&*&*&*\\
0&0&0&*\end{array}\right]\right\}.$$
But these two realisations are equivalent under conjugation by
\begin{equation}\label{sigma1}\setlength{\arraycolsep}{3pt}
\renewcommand{\arraystretch}{.8}
\left[\begin{array}{cccc}
0&1&0&0\\
1&0&0&0\\
0&0&1&0\\
0&0&0&1\end{array}\right]\end{equation}
and, as explained in~\cite{me,EW}, the effect of this conjugation is that the 
formula for pulling back a homogeneous vector bundle from $Z$ to 
${\mathbb{X}}$ includes the action of the Weyl group element represented 
by~(\ref{sigma1}). Specifically,
$$\mu^*(a\,|\,b,c\,|\,d)=(b\,\|\,a\,|\,c\,|\,d)+\cdots$$
and, in particular,
\begin{equation}\label{pullback_of_our_line_bundle}
\mu^*(1\,|\,0,0\,|\,0)=(0\,\|\,1\,|\,0\,|\,0).\end{equation}
This bundle on ${\mathbb{X}}$ makes its effect felt in modifying the spectral
sequence~(\ref{SS}) as
$$E_1^{p,q}=\Gamma(M,\tau_*^q(\Lambda_\mu^{p,0}
     \otimes(0\,\|\,1\,|\,0\,|\,0)|_X))\Longrightarrow
H^{p+q}(\Gamma(X,\Lambda_X^{0,\bullet}
     \otimes(0\,\|\,1\,|\,0\,|\,0)|_X),\bar\partial_X)$$
and also the spectral sequence (\ref{topology}) as applied in proving
Proposition~\ref{computation_involutive_cohomology}. In fact, since
$H^r({\mathbb{F}}_{1,3}({\mathbb{C}}^4),{\mathcal{O}}(1\,|\,0,0\,|\,0))=0$ for
all~$r$ (as a particular instance of the Bott-Borel-Weil Theorem~\cite{bott}), 
following the proof of Proposition~\ref{computation_involutive_cohomology} 
demonstrates the following.
\begin{proposition}
Concerning the double fibration~\eqref{balanced}, we have
$$H^r(\Gamma(X,\Lambda_X^{0,\bullet}
               \otimes(0\,\|\,1\,|\,0\,|\,0)|_X)),\bar\partial_X)=0
\enskip\forall\,r.$$
\end{proposition}
\noindent Therefore, the spectral sequence
$$E_1^{p,q}=\Gamma(M,\tau_*^q(\Lambda_\mu^{p,0}
     \otimes(0\,\|\,1\,|\,0\,|\,0)|_X))$$
converges to zero. It remains to compute the bundles involved and for this we 
may proceed as before, instead computing
$$\nu_*^q(\Lambda_\mu^{p,0}\otimes(0\,\|\,1\,|\,0\,|\,0))\quad\mbox{for }
     \nu:{\mathbb{X}}\to{\mathbb{M}}$$
and then restricting to ${\mathbb{CP}}_3=M\hookrightarrow{\mathbb{M}}$. This 
is a matter of combining (\ref{pullback_of_our_line_bundle}) with 
(\ref{key_bundle}) and applying the Bott-Borel-Weil Theorem as formulated 
in~\cite{mge}. 
\begin{proposition} The direct images 
$\nu_*^q(\Lambda_\mu^{p,0}\otimes(0\,\|\,1\,|\,0\,|\,0))$ vanish for $q\geq 1$ 
and for $q=0$ are as follows
\begin{equation}\label{table}\begin{array}{|c|c|c|c|c|}
p=0&p=1&p=2&p=3&p=4\\ \hline
\rule{0pt}{10pt}
0&0&(-2\,\|\,1,1,1)&(-1\,\|\,0,1,1)\oplus(1\,\|\,0,0,0)&(0\,\|\,0,0,1)
\end{array}\,.\end{equation}
\end{proposition}
\begin{proof}
According to the Bott-Borel-Weil Theorem, some particular direct images are
$$\begin{array}{rcl}
\nu_*(a\,\|\,b\,|\,c\,|\,d)&=&(a\,\|\,b,c,d)\enskip\mbox{if }b\leq c\leq d\\
\nu_*^1(a\,\|\,b\,|\,c\,|\,d)&=&(a\,\|\,b+1,c-1,d)\enskip\mbox{if }
b+1\leq c-1\leq d\\
\nu_*^1(a\,\|\,b\,|\,c\,|\,d)&=&(a\,\|\,b,d+1,c-1)\enskip\mbox{if }
b\leq d+1\leq c-1\\
\end{array}$$
and, in these cases, all other direct images vanish. Furthermore, 
$(a\,\|\,b\,|\,c\,|\,d)$ has all direct images vanishing if any two of 
$b,c+1,d+2$ coincide. This will be sufficient for our purposes. In particular,
$$\nu_*^q(\Lambda_\mu^{0,0}\otimes(0\,\|\,1\,|\,0\,|\,0))=
\nu_*^q(0\,\|\,1\,|\,0\,|\,0)=0\enskip\forall\,q.$$

Next,
$$\Lambda_\mu^{1,0}\otimes(0\,\|\,1\,|\,0\,|\,0)=
\begin{array}{c}
(-1\,\|\,1\,|\,0\,|\,1)+(-1\,\|\,1\,|\,1\,|\,0)\\ \oplus\\
(1\,\|\,0\,|\,0\,|\,0)+(1\,\|\,1\,|\,{-1}\,|\,0)
\end{array}$$
so
\begin{equation}\label{need_to_compute}
\nu_*^q(\Lambda_\mu^{1,0}\otimes(0\,\|\,1\,|\,0\,|\,0))=
\nu_*^q((1\,\|\,0\,|\,0\,|\,0)+(1\,\|\,1\,|\,{-1}\,|\,0)).\end{equation}
This requires further work since we need to know the connecting 
homomorphism
\begin{equation}\label{connecting}\begin{array}{ccc}
\nu_*(1\,\|\,0\,|\,0\,|\,0)&\to&\nu_*^1(1\,\|\,1\,|\,{-1}\,|\,0)\\
\|&&\|\\
(1\,\|\,0,0,0)&\xrightarrow{\,?\,}&(1\,\|\,0,0,0)
\end{array}\end{equation}
induced by this extension. For this, we consult again the 
diagram~(\ref{final_basic_diagram}), finding that the bottom row in case of 
(\ref{homogeneous}) with $n=3$ is
$$\setlength{\arraycolsep}{2pt}
\begin{array}{ccccccccl}
0&\to&{\mathcal{C}}^{1,0}&\to&\nu^*\Lambda_{\mathbb{M}}^{1,0}&\to&
\Lambda_\mu^{1,0}&\to&0\\
&&\|&&\|&&\|\\
0&\to&\begin{array}{c}
(-1\,\|\,1\,|\,0\,|\,0)\\
\oplus\\
(1\,\|\,0\,|\,0\,|\,{-1})
\end{array}&\to&
\nu^*\left[\begin{array}{c}(-1\|0,0,1)\\ \oplus\\ (1\|{-1},0,0)\end{array}\right]
&\to&\begin{array}{c}
(-1\,\|\,0\,|\,0\,|\,1)+(-1\,\|\,0\,|\,1\,|\,0)\\ \oplus\\
(1\,\|\,{-1}\,|\,0\,|\,0)+(1\,\|\,0\,|\,{-1}\,|\,0)
\end{array}&\to& 0\,,\end{array}$$
which, in particular, yields the short exact sequence
$$0\to(1\,\|\,0\,|\,0\,|\,{-1})\to\nu^*(1\|{-1},0,0)\to
(1\,\|\,{-1}\,|\,0\,|\,0)+(1\,\|\,0\,|\,{-1}\,|\,0)\to 0$$
and, therefore, when tensored with $(0\,\|\,1\,|\,0\,|\,0)$ the short exact 
sequence
$$0\to(1\,\|\,1\,|\,0\,|\,{-1})\to
\nu^*(1\|{-1},0,0)\otimes(0\,\|\,1\,|\,0\,|\,0)\to
(1\,\|\,0\,|\,0\,|\,0)+(1\,\|\,1\,|\,{-1}\,|\,0)\to 0$$
from which it follows that all the direct images (\ref{need_to_compute}) vanish
(equivalently, that the connecting homomorphism (\ref{connecting}) is an 
isomorphism, as one might expect).

Next, we should compute the direct images of 
$\Lambda_\mu^{2,0}\otimes(0\,\|\,1\,|\,0\,|\,0)$, i.e.~of
$$(-2\,\|\,1\,|\,1\,|\,1)\oplus
\bigg[(0\,\|\,0\,|\,0\,|\,1)+
\!\!\!\begin{array}{c}
(0\,\|\,1\,|\,{-1}\,|\,1)\\ \oplus\\
(0\,\|\,0\,|\,1\,|\,0)
\end{array}\!\!\!
+(0\,\|\,1\,|\,0\,|\,0)\bigg]
\oplus (2\,\|\,0\,|\,{-1}\,|\,0).$$
The induced connecting homomorphism 
$\nu_*(0\,\|\,0\,|\,0\,|\,1)\to\nu_*^1(0\,\|\,1\,|\,{-1}\,|\,1)$ is again an 
isomorphism by similar reasoning and only $(-2\,\|\,1\,|\,1\,|\,1)$ 
contributes to the direct images, as claimed in~(\ref{table}).

Next, 
$$\Lambda_\mu^{3,0}\otimes(0\,\|\,1\,|\,0\,|\,0)=
\begin{array}{c}
(-1\,\|\,0\,|\,1\,|\,1)+(-1\,\|\,1\,|\,0\,|\,1)\\ \oplus\\
(1\,\|\,0\,|\,{-1}\,|\,1)+(1\,\|\,0\,|\,0\,|\,0)
\end{array}$$
and, finally,
$$\Lambda_\mu^{4,0}\otimes(0\,\|\,1\,|\,0\,|\,0)=(0\,\|\,0\,|\,0\,|\,1)$$
from which the rest of (\ref{table}) is immediate.
\end{proof}

\noindent Assembling these various computations yields the following.
\begin{theorem}\label{assembled_theorem}
There is an elliptic and globally exact complex on\/~${\mathbb{CP}}_3$
\begin{equation}\label{globally_exact}
\setlength{\arraycolsep}{2pt}\begin{array}{ccccccccc}
0&\to&(-2\,\|\,1,1,1)&\to&
(-1\,\|\,0,1,1)
&\to&(0\,\|\,0,0,1)&\to&0.\\ &&&&\oplus&\nearrow\\ 
&&&&(1\,\|\,0,0,0)\end{array}\end{equation}
\end{theorem}
\begin{proof} Everything is shown save for the following two observations.
Firstly, there is no possible first order differential operator
$(-2\,\|\,1,1,1)\to(1\,\|\,0,0,0)$ since there is no possible 
${\mathrm{SU}}(4)$-invariant symbol. Indeed, from (\ref{Lambda_M^1}), we have
$$\Lambda_M^1\otimes(-2\,\|\,1,1,1)=(-3\,\|\,1,1,2)\oplus
(-1\,\|\,0,1,1).$$
Similar symbol considerations
$$\begin{array}{rcl}
\Lambda_M^1\otimes(-1\,\|\,0,1,1)&=&(-2\,\|\,0,1,2)\oplus(-2\,\|\,1,1,1)
\oplus(0\,\|\,{-1},1,1)\oplus(0\,\|\,0,0,1)\\
\Lambda_M^1\otimes(1\,\|\,0,0,0)&=&(0\,\|\,0,0,1)\oplus(2\,\|\,{-1},0,0)
\end{array}$$
also allow one to check that the complex is elliptic. 
\end{proof}

Invariance under ${\mathrm{SU}}(4)$ identifies the operators
explicitly. Specifically, if we denote by $L$ the homogeneous line bundle 
$L=(-2\,\|\,1,1,1)$, then (\ref{globally_exact}) becomes
$$\begin{array}{ccccccccc}
0&\to&L&\xrightarrow{\,\partial\,}&
\Lambda_M^{1,0}\otimes L
&\xrightarrow{\,\partial\,}&\Lambda_M^{2,0}\otimes L&\to&0.\\ &&&&\oplus&&
\makebox[0pt][l]{$\downarrow\!\!\rule[3pt]{.5pt}{4pt}\;
{\scriptstyle\kappa\wedge\kappa\wedge\underbar{\;}}$}\\ 
&&&&\Lambda_M^{3,0}\otimes L&\xrightarrow{\,\bar\partial}
&\Lambda_M^{3,1}\otimes L\end{array}$$
As a check, Theorem~\ref{assembled_theorem} says that $L$ has no global
anti-holomorphic sections and this is certainly true because its complex
conjugate $(2\,\|\,{-1},{-1},{-1})$ has no global holomorphic sections (it is
the homogeneous holomorphic bundle $(2\,|\,{-1},{-1},{-1})$ in the notation
of~\cite{mge}, which has singular infinitesimal character).

Other homogeneous holomorphic bundles on the twistor space
$Z={\mathbb{F}}_{1,n}({\mathbb{C}}^{n+1})$ will give rise to other invariant
complexes of differential operators on~${\mathbb{CP}}_n$. The author suspects
that the holomorphic tangent bundle $\Theta$ will give rise to an especially 
interesting complex (since $H^1(Z,\Theta)$ parameterises the 
infinitesimal deformations of~$Z$ as a complex manifold). Unfortunately, he 
has not yet been able to complete the calculation in this case.

\section{Another particular transform}\label{flag}
This section is concerned with an instance of the holomorphic double fibration
transform as formulated in general in~\cite{FHW}. Specifically, let us consider
the complex flag manifold $Z={\mathbb{F}}_{1,n}({\mathbb{C}}^{n+1})$ under the
action of ${\mathrm{SU}}(n,1)$. There are three open orbits for this action,
easily described in terms of geometry in~${\mathbb{CP}}_n$. The orbits of
${\mathrm{SU}}(n,1)$ acting on ${\mathbb{CP}}_n$ are the open ball $B$, its
boundary, and the complement of its closure. As in \S\ref{intro}
and~\S\ref{particular}, an element $(L,H)$ in $Z$ may be viewed as a point on
a hyperplane in~${\mathbb{CP}}_n$. The three open orbits are given by the 
following restrictions.
\begin{itemize}
\item the point $L$ lies in the ball~$B$,
\item the hyperplane $H$ lies outside the ball $B$,
\item the point $L$ lies outside $B$ but the hyperplane $H$ intersects~$B$.
\end{itemize}
The set of hyperplanes lying outside $B$ defines an open subset in the dual
projective space~${\mathbb{CP}}_n^*$, which we may identify with~$\bar{B}$, 
i.e.~the ball $B$ with its conjugate complex structure. 

Let us consider the third of the options above for the open orbits of
${\mathrm{SU}}(n,1)$ acting on $Z$ and call it~$D$. By definition it is a 
{\em flag domain}\/. Following the notation of~\cite{FHW}, its 
{\em cycle space\/} ${\mathcal{M}}_D$ is $B\times\bar{B}$ inside
${\mathbb{CP}}_n\times{\mathbb{CP}}_n^*$ and the correspondence space 
${\mathfrak{X}}_D$ is exactly $\nu^{-1}({\mathcal{M}}_D)$ for the complexified 
correspondence of~\S\ref{particular}. Thus, we have an open inclusion
$$\raisebox{-20pt}{\begin{picture}(55,40)
\put(0,5){\makebox(0,0){$D$}}
\put(30,35){\makebox(0,0){${\mathfrak{X}}_D$}}
\put(60,5){\makebox(0,0){${\mathcal{M}}_D$}}
\put(25,30){\vector(-1,-1){18}}
\put(35,30){\vector(1,-1){18}}
\put(12,24){\makebox(0,0){$\mu$}}
\put(48,24){\makebox(0,0){$\nu$}}
\end{picture}}
\qquad\mbox{\LARGE${}^{\scriptscriptstyle\mathrm{open}}\!\hookrightarrow$}
\hspace{30pt}
\raisebox{-20pt}{\begin{picture}(55,40)
\put(0,5){\makebox(0,0){${\mathbb{F}}_{1,n}({\mathbb{C}}^{n+1})$}}
\put(30,35){\makebox(0,0){${\mathbb{X}}$}}
\put(55,5){\makebox(0,0)[l]{${\mathbb{M}}$}}
\put(25,30){\vector(-1,-1){18}}
\put(35,30){\vector(1,-1){18}}
\put(12,24){\makebox(0,0){$\mu$}}
\put(48,24){\makebox(0,0){$\nu$}}
\end{picture}}$$
and, in particular, the fibres of $\nu$ over ${\mathcal{M}}_D$ coincide
with the fibres of $\nu:{\mathbb{X}}\to{\mathbb{M}}$ restricted
to~${\mathcal{M}}_D$. The spectral sequence~\cite{BE,EW} for the resulting double fibration
transform starting with a holomorphic vector bundle $E$ on $D$ reads
\begin{equation}\label{usualSS}
E_1^{p,q}=\Gamma({\mathcal{M}}_D,\nu_*^q(\Lambda_\mu^{p,0}\otimes\mu^*E))
\Longrightarrow H^{p+q}(D,{\mathcal{O}}(E))\end{equation}
under the assumptions that ${\mathcal{M}}_D$ is Stein and that
$\mu:{\mathfrak{X}}_D\to D$ has contractible fibres, both of which are true 
for any flag domain~\cite{FHW} and directly seen to be the case here. 

Observe that the terms in this spectral sequence (when $E$ is trivial) are
almost the same as in~(\ref{SS}). Certainly, the direct image bundles may be
obtained by working on the homogeneous correspondence~(\ref{homogeneous}) and
then restricting to~${\mathcal{M}}_D$. As our final example, let us carry this
out for $n=3$ and for $E$ being the canonical bundle on~$D$. This is precisely
the restriction to $D$ of the homogeneous line bundle $(3\,|\,0,0\,|\,{-3})$
on~$Z$. Following exactly the procedures of~\S\ref{particular} we find the
following homogeneous bundles for
$\Lambda_\mu^{p,0}\otimes\mu^*(3\,|\,0,0\,|\,{-3})=
\Lambda_\mu^{p,0}\otimes(0\,\|\,3\,|\,0\,|\,{-3})$.\footnotesize
$$\begin{array}{r|l}p=0&(0\,\|\,3\,|\,0\,|\,{-3})\\ \hline
p=1&\hspace{-5pt}\begin{array}{c}
(-1\,\|\,3\,|\,0\,|\,{-2})+(-1\,\|\,3\,|\,1\,|\,{-3})\\ \oplus\\
(1\,\|\,2\,|\,0\,|\,{-3})+(1\,\|\,3\,|\,{-1}\,|\,{-3})\hspace{5pt}
\end{array}\\ \hline
p=2&(-2\,\|\,3\,|\,1\,|\,{-2})\oplus
\bigg[(0\,\|\,2\,|\,0\,|\,{-2})+
\!\!\!\begin{array}{c}
(0\,\|\,3\,|\,{-1}\,|\,{-2})\\ \oplus\\
(0\,\|\,2\,|\,1\,|\,{-3})
\end{array}\!\!\!
+(0\,\|\,3\,|\,0\,|\,{-3})\bigg]
\oplus (2\,\|\,2\,|\,{-1}\,|\,{-3})\\ \hline
p=3&\hspace{-5pt}\begin{array}{c}
(-1\,\|\,2\,|\,1\,|\,{-2})+(-1\,\|\,3\,|\,0\,|\,{-2})\\ \oplus\\
(1\,\|\,2\,|\,{-1}\,|\,{-2})+(1\,\|\,2\,|\,0\,|\,{-3})\hspace{5pt}
\end{array}\\ \hline
p=4&(0\,\|\,2\,|\,0\,|\,{-2})
\end{array}$$\normalsize
Using the Bott-Borel-Weil Theorem, as formulated in~\cite{mge}, we find that 
the only non-zero direct images 
$\nu_*^q\big(\Lambda_\mu^{p,0}\otimes\mu^*(3\,|\,0,0\,|\,{-3})\big)$ are when 
$q=3$ as follows.\footnotesize
$$\begin{array}{r|l}p=0&(0\,\|\,{-1},0,1)\\ \hline
p=1&\hspace{-5pt}\begin{array}{c}
(-1\,\|\,0,0,1)\oplus(-1\,\|\,{-1},1,1)\\ \oplus\hspace{6.8pt}\\
(1\,\|\,{-1},0,0)\oplus(1\,\|\,{-1},{-1},1)
\end{array}\\ \hline
\rule[-4pt]{0pt}{12pt}p=2&(-2\,\|\,0,1,1)\oplus
\big[(0\,\|\,0,0,0)\oplus(0\,\|\,{-1},0,1)\big]
\oplus (2\,\|\,{-1},{-1},0)\\ \hline
p=3&\hspace{-5pt}\begin{array}{c}
(-1\,\|\,0,0,1)\\ \oplus\\
(1\,\|\,{-1},0,0)
\end{array}\\ \hline
p=4&(0\,\|\,0,0,0) \end{array}$$\normalsize 
We conclude, for example, from (\ref{usualSS}) that
$H^3(D,{\mathcal{O}}(3\,|\,0,0\,|\,{-3}))$ is realised on 
$${\mathcal{M}}_D=B\times\bar{B}\subset{\mathbb{CP}}_3\times{\mathbb{CP}}_3^*$$
as the kernel of the holomorphic differential operator
$$(0\,\|\,{-1},0,1)\begin{picture}(32,0)(0,-3)
\put(15,10){\makebox(0,0){$\bar\eth$}}
\put(15,-10){\makebox(0,0){$\eth$}}
\put(5,1){\vector(3,1){25}}
\put(5,-1){\vector(3,-1){25}}
\end{picture}
\setlength{\arraycolsep}{2pt}\begin{array}{ccc}
(-1\,\|\,0,0,1)&\oplus&(-1\,\|\,{-1},1,1)\\ &\oplus\\
(1\,\|\,{-1},0,0)&\oplus&(1\,\|\,{-1},{-1},1)
\end{array}$$
where $\eth$ (respectively~$\bar\eth$) denotes holomorphic differentiation in the 
direction of ${\mathbb{CP}}_3$ (respectively~${\mathbb{CP}}_3^*$) followed by 
projection to the indicated bundles:
$$\begin{array}{rcl}(0\,\|\,{-1},0,1)&\longrightarrow&
\Lambda_{{\mathbb{CP}}_3}^{1,0}\otimes(0\,\|\,{-1},0,1)\\
&=&(1\,\|\,{-1},0,0)\otimes(0\,\|\,{-1},0,1)\\
&=&(1\,\|\,{-2},0,1)\oplus(1\,\|\,{-1},{-1},1)\oplus(1\,\|\,{-1},0,0)\\
&\twoheadrightarrow&(1\,\|\,{-1},{-1},1)\oplus(1\,\|\,{-1},0,0).
\end{array}$$

It is interesting to note that the entire complex 
$\nu_*^3\big(\Lambda_\mu^{\bullet,0}\otimes\mu^*(3\,|\,0,0\,|\,{-3})\big)$ on 
${\mathbb{M}}$ is the analytic continuation of the elliptic complex
$$0\to\Lambda_\perp^{1,1}\begin{array}{c}\nearrow\\ \searrow\end{array}
\begin{array}{c}\Lambda^{1,2}\\ \oplus\\ \Lambda^{2,1}\end{array}
\begin{array}{c}\nearrow\\ \searrow\\ \nearrow\\ \searrow\end{array}
\begin{array}{c}\Lambda^{1,3}\\ \oplus\\ \Lambda^{2,2}\\ \oplus\\ 
\Lambda^{3,1}\end{array}
\begin{array}{c}\searrow\\ \nearrow\\ \searrow\\ \nearrow\end{array}
\begin{array}{c}\Lambda^{2,3}\\ \oplus\\ \Lambda^{3,2}\end{array}
\begin{array}{c}\searrow\\ \nearrow\end{array}\Lambda^{3,3}\to 0$$
on~$M={\mathbb{CP}}_3$ and that this complex is the formal adjoint 
of~(\ref{elliptic_complex}), exactly as predicted by duality 
\cite[Theorem~4.1]{EW}. The transform described in this section is an 
example of the much more general theory developed in~\cite{EW_in_prep}.

\bibliographystyle{amsplain}

\end{document}